\documentclass[a4paper, 11pt]{article}
\usepackage{amsfonts, amssymb, amsmath, amsthm, latexsym,graphicx}

\begin{document}
\baselineskip=22pt

\theoremstyle{plain}
\newtheorem{thm}{Theorem}
\newtheorem{pro}[thm]{Proposition}
\newtheorem{cor}[thm]{Corollary}
\newtheorem{con}[thm]{Conjecture}
\newtheorem{lem}[thm]{Lemma}

\theoremstyle{definition}
\newtheorem{prob}[thm]{Problem}
\newtheorem{rem}[thm]{Remark}
\newtheorem{example}[thm]{Example}

\newcommand{\m}{{\rm mult}}
\newcommand{\la}{\lambda}
\newcommand{\im}{{\rm Im}}
\newcommand{\x}{{\bf x}}
\newcommand{\w}{{\bf w}}
\newcommand{\bu}{{\bf u}}
\newcommand{\bg}{{\rm BG}}

\title{\bf  Cyclic decomposition of $k$-permutations \\and  eigenvalues of the arrangement graphs}

\author{\large Bai Fan Chen$^{\,\rm 1}$ \quad \quad  Ebrahim Ghorbani$^{\,\rm 2,3}$  \quad \quad Kok Bin Wong$^{\,\rm 1}$\\[.4cm]
{\sl $^{\rm 1}$ Institute of Mathematical Sciences, University of Malaya,}\\
{\sl 50603 Kuala Lumpur, Malaysia}\\[0.3cm]
{\sl $^{\rm 2}$Department of Mathematics, K.N. Toosi University of Technology,}\\
{\sl P.O. Box 16315-1618, Tehran, Iran}\\[0.3cm]
{\sl $^{\rm 3}$School of Mathematics, Institute for Research in Fundamental
Sciences (IPM),}\\{\sl P.O. Box
19395-5746, Tehran, Iran }
\\[0.5cm]{
$\mathsf{tufofo1120@gmail.com}$ \quad\quad  $\mathsf{e\_ghorbani@ipm.ir}$ \quad\quad  $\mathsf{kbwong@um.edu.my}$}}

 \maketitle

\begin{abstract}\noindent
 The $(n,k)$-arrangement graph  $A(n,k)$ is a graph with all the $k$-permutations of an $n$-element set as vertices where two $k$-permutations are adjacent if they agree in exactly $k-1$ positions.
We introduce a cyclic decomposition for $k$-permutations and show that this gives rise to a very fine equitable partition
of $A(n,k)$. This equitable partition can be employed to compute the complete set of eigenvalues (of the adjacency matrix) of $A(n,k)$. Consequently, we determine the eigenvalues of $A(n,k)$ for small values of $k$.
Finally, we show that any eigenvalue of the Johnson graph $J(n,k)$ is an eigenvalue of $A(n,k)$ and that $-k$ is the smallest eigenvalue of $A(n,k)$ with multiplicity ${\cal O}(n^k)$ for fixed $k$.

\vspace{5mm}
\noindent {\it AMS Mathematics Subject Classification\,(2010):}   05A05, 05C50
\end{abstract}

\section{Introduction}

Let $G$ be a simple graph with $\nu$ vertices. The {\em adjacency matrix}  of   $G$ is a $\nu\times\nu$ matrix where its rows and
columns are indexed by  the vertex set of $G$  and its $(u, v)$-entry  is $1$ if the vertices $u$ and
$v$ are adjacent and $0$ otherwise. By the {\em eigenvalues} of $G$ we mean the eigenvalues of its adjacency matrix.

For a positive integer $\ell$, let $[\ell]:=\{1,\ldots,\ell\}$. For positive integers $k,n$ with $k\le n$, a {\em $k$-permutation} of $[n]$ is an injective function from $[k]$ to $[n]$. When $k=n$, a $k$-permutation is a permutation.
So any $k$-permutation $\pi$ can be represented by a vector $(i_1,\ldots,i_k)$ where $\pi(j)=i_j$ for $j=1,\ldots,k$.
We denote the set of all $k$-permutations of $[n]$ by $V(n,k)$.
The set $\{i_1,\ldots,i_k\}$ is the {\em image} of $\pi$ and denoted $\im(\pi)$.
Unlike permutations which have a decomposition into cycles, $k$-permutations in general do not have such a decomposition.
We observe that a decomposition of a $k$-permutation into cycles and `paths' is possible
where paths are the same as cycles except that in a path the last element is not mapped to the first element.
We then define the cycle type for a $k$-permutation $\pi$ to be the list of integers consisting of the lengths of the cycles and the paths appeared in the decomposition of $\pi$. This gives rise to the `cycle-type partition' of $V(n,k)$ in which each class consists of all elements of $V(n,k)$ sharing the same cycle type. The reason for studying the cycle-type partition of $k$-permutations will become clear below.

The {\em $(n,k)$-arrangement graph}  $A(n,k)$ has $V(n,k)$ as its vertices, and two $k$-permutations $\pi=(u_1,\ldots,u_k)$ and $\rho=(v_1,\ldots,v_k)$ are adjacent if they agree in exactly $k-1$ positions, i.e. if for exactly one $i_0$, $u_{i_0}\ne v_{i_0}$ and for all $i\ne i_0$, $u_i=v_i$.
The family of arrangement graphs was first introduced in \cite{dt} as an interconnection network model for parallel computation.
In the interconnection network model, each processor has its own memory unit and
communicates with the other processors through a topological network, i.e. a graph.
For this purpose, the arrangement graphs possess many nice properties such as having small diameter, a hierarchical structure, vertex and edge symmetry, simple shortest path routing, high connectivity, etc. Many properties of arrangement graphs have been studied by a number of authors, see, e.g. \cite{cll,cly,cgq,cqs,cc,ttt,zx}.
Another family of graphs with the same nature as the arrangement graphs are the derangement graphs.
 The {\em $n$-derangement graph} is a graph whose vertices correspond to all the permutations of $[n]$ where two permutations are adjacent if they differ in all $n$ positions. It is known that the eigenvalues of the derangement graph are integers (see \cite{Babai, DS, Lub, Ram}). For other properties of the eigenvalues of the derangement graph, we refer the reader to \cite{Ku-Wales, Ku-Wong, Renteln}.

As an application of the cycle-type partition of $V(n,k)$, we consider the problem of determining the eigenvalues of the arrangment graphs. It turns out that the cycle-type partition of $V(n,k)$ is indeed an equitable partition of the graph $A(n,k)$.
Normally, the eigenvalues of equitable partitions of a graph give a subset of the set of eigenvalues of the graph.
However, in view of a result of Godsil and McKay \cite{gm}, the cycle-type partition of $V(n,k)$ is fine enough to give the complete set of eigenvalues as well as their multiplicities. Consequently, we will be able to determine the eigenvalues of $A(n,k)$ for small values of $k$. We also show that any eigenvalue of the  Johnson graph $J(n,k)$ is an eigenvalue of $A(n,k)$. (Recall that the  Johnson graph $J(n,k)$ has all the $k$-subsets of $[n]$ as vertices where two $k$-subsets are adjacent if they intersect in exactly $k-1$ elements.) Finally, we prove that that $-k$ is the smallest eigenvalue of $A(n,k)$ with multiplicity ${\cal O}(n^k)$ for fixed $k$. We shall close the paper by some open problems on eigenvalues of the arrangement graphs.

\section{Cyclic decomposition of a $k$-permutation and basic graphs}

A permutation can be decomposed into disjoint cycles.
However, in general such a decomposition does not exist for a $k$-permutation.
In this section, we demonstrate  how a decomposition of $k$-permutations is possible by taking into account a relaxation on cycles.

\subsection{Decomposition of $k$-permutation into cycles and paths}

We call a $k$-permutation $\pi$ a {\em path} of length $\ell$ if
for some $u_1,\ldots,u_\ell$ and $v$, $\pi(u_t)=u_{t+1}$ for $t=1,\ldots,\ell-1$ and $\pi(u_\ell)=v$ such that
$u_t\in[k]$ for all  $t=1,\ldots,\ell$ and $v\not\in[k]$.
We denote such a path $\pi$ by $(u_1\ldots u_\ell\,v]$.
As usual, $(u_1\ldots u_\ell)$ denotes a cycle of length $\ell$.
The same method for decomposing permutations into disjoint cycles can be employed for decomposition of $k$-permutations into disjoint cycles and paths. We call this decomposition the {\em cyclic decomposition} of $k$-permutations.
Here is examples of decompositions of some 5-permutations:
\begin{align*}
(1,2,3,i,j)&=(1)(2)(3)(4\,i](5\,j],\\
(2,3,4,i,j)&=(1\,2\,3\,4\,i](5\,j],\\
(2,i,j,5,4)&=(1\,2\,i](3\,j](4\,5),
\end{align*}
where $i,j>5$. 

We put this observation formally in the following proposition. The straightforward proof is similar to the case of permutations.

\begin{pro} Any $k$-permutation is a product of disjoint cycles and paths. This decomposition is unique up to the order in which the cycles and paths are written.
\end{pro}

In some applications, it would be useful to consider a graphical representation of a $k$-permutation $\pi$.
This can be done simply by constructing a (directed) graph with vertices $[k]\cup\im(\pi)$ and with the set of arcs $\{(u,\pi(u))\mid u\in[k]\}$. We call the resulting graph, the {\em basic graph} of $\pi$ and denote it by $\bg(\pi)$.
It is seen that $\bg(\pi)$ consists of a union of directed cycle graphs and path graphs where the cycle and path graphs correspond to the cycles and paths in the cyclic decomposition of $\pi$, respectively.
It turns out that in $\bg(\pi)$   all the edges in the cycles and paths  have the same directions, the vertices  of $[k]\cap\im(\pi)$ have degree two, other vertices have degree one, and that
 any vertex of $\im(\pi)\setminus[k]$ is the head of a (directed) path.
In Table~\ref{part3}, all non-isomorphic basic graphs of 3-permutations are depicted where the  edges are shown without directions and the black and white vertices represent elements of $[k]$ and $\im(\pi)\setminus[k]$, respectively.
\begin{table}
  \centering
  \begin{tabular}{cllc}
\hline
  $3$-permutation & Decomposition & Partition of 3& Basic graph\\
   \hline
  $(1,2,3)$ & $(1)(2)(3)$ & $111$ &\includegraphics[width=1.6cm]{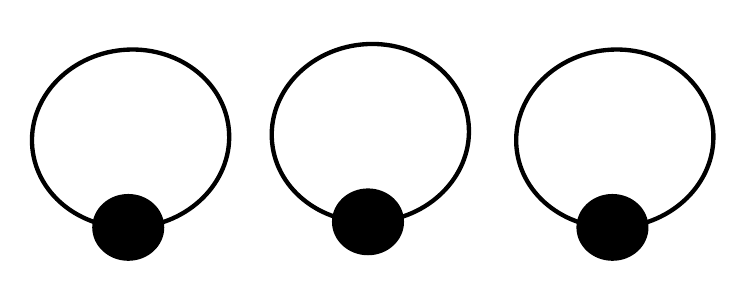}\\
  $(1,3,2)$ & $(1)(2\,3)$& $12$ &\includegraphics[width=1.5cm]{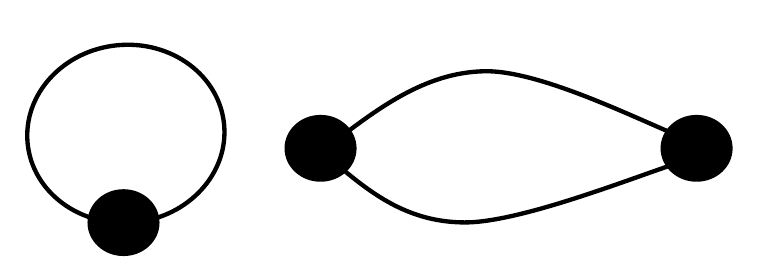}\\
   $(2,3,1)$ & $(1\,2\,3)$& $3$&\includegraphics[width=1.2cm]{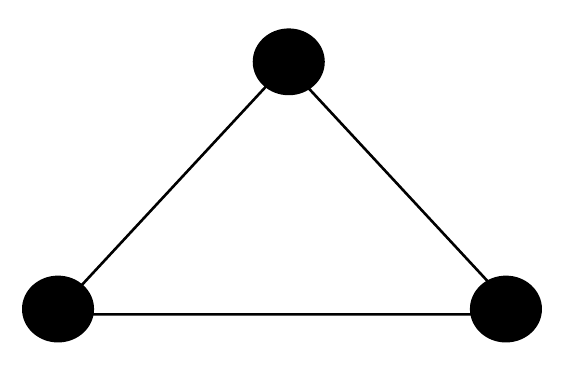}\\
  $(1,2,i)$ & $(1)(2)(3\,i]$& $111'$ &\includegraphics[width=1.4cm]{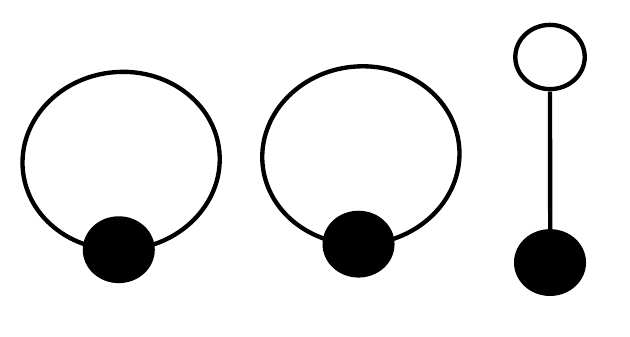}\\
  $(1,3,i)$ & $(1)(2\,3\,i]$& $12'$ &\includegraphics[width=1.1cm]{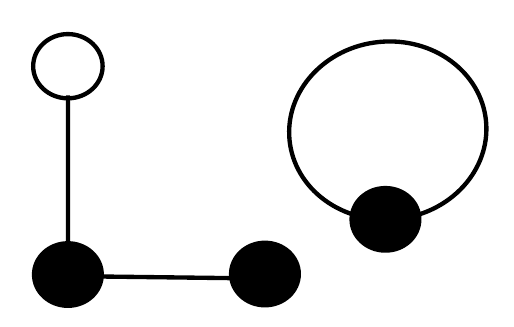} \\
  $(2,3,i)$ & $(1\,2\,3\,i]$& $3'$ &\includegraphics[width=.9cm]{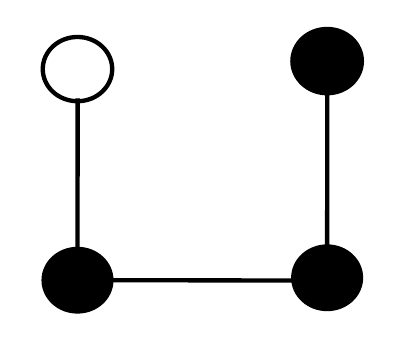}\\
  $(2,1,i)$& $(1\,2)(3,\,i]$& $21'$&\includegraphics[width=1.4cm]{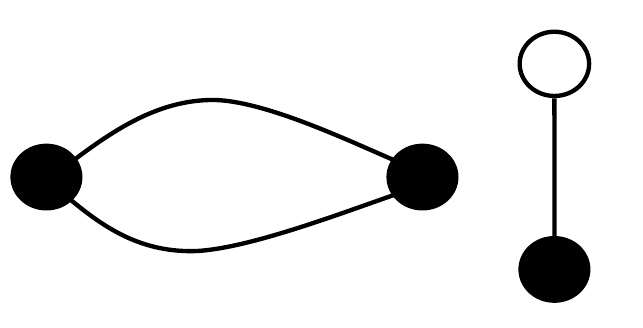} \\
  $(1,i,j)$& $(1)(2\,i](3\,j]$& $11'1'$ &\includegraphics[width=1.25cm]{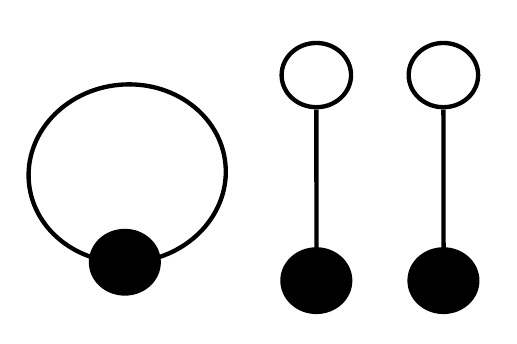}\\
  $(2,i,j)$& $(1\,2\,i](3\,j]$& $1'2'$&\includegraphics[width=1cm]{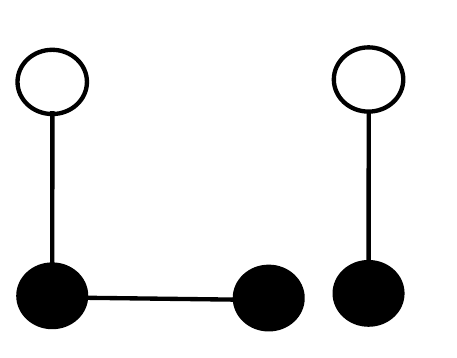}\\
 $(i,j,\ell)$& $(1\,i](2\,j](3\,\ell]$& $1'1'1'$& \includegraphics[width=.9cm]{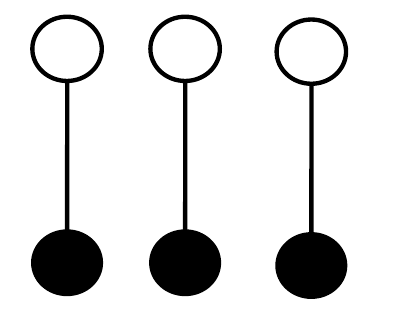} \\
  \hline
\end{tabular}
  \caption{ The cycle structures of 3-permutations (here $i,j,\ell>3$) 
}\label{part3}
\end{table}

We usually consider basic graphs without directions as the directions of the edges has no effect on the isomorphism type of them. However, in some applications such as the proof of Theorem~\ref{qutiont} the direction of the edges might be useful.
We remark that any (multi-)graph with $k$ edges, maximum degree at most $2$ and with no isolated vertices is a basic graph of some $k$-permutation. This is another description of the family of basic graphs.

\subsection{Cycle structure and partitions of $k$ into parts of two kinds}

To any cyclic decomposition of a permutation $\sigma$ one may assign a list of integers consisting of the lengths of the cycles appearing in the decomposition.
This list is called the {\em cycle structure} of $\sigma$. There is a one-to-one correspondence between the cycle structure of  permutations on $[n]$ and
the partitions of the integer $n$. For the case of $k$-permutations, we need to distinguish between  cycles and paths.
To this end, we take into account
 the partitions of $k$ into parts of two kinds. More precisely,  assume that there are integers of two kinds $r$ and $r'$ and we consider the ways to write $n$ as
a sum of integers of either kind where the order of terms in the
sum does not matter.  For instance, $k=2$ has the following partitions into parts of two kinds: $11$, $11'$, $1'1'$, $2$, and $2'$.
We define the {\em cycle structure} or {\em cycle type} of a $k$-permutation $\pi$ to be the list consisting of the lengths of the cycles and the paths appearing in the cyclic decomposition of $\pi$ where for cycles we write integers of the first kind and for paths we write integers of the second kind.
For example, any $3$-permutation of $[n]$ with $n\ge6$ has one of the ten different types represented in Table~\ref{part3} where
for each type, an instance of a $3$-permutation with that type together with its cyclic decomposition and the respective basic graph are demonstrated.

\begin{pro} There is a one-to-one correspondence between the cycle structure of  $k$-permutations and
the partitions of $k$ into parts of two kinds.
\end{pro}

\begin{rem}\label{c(k)} We define $c(k)$ to be the number of partitions of $k$ into parts of two kinds.
The sequence consisting of  $c(k)$'s, namely the sequence
  $1, 2, 5, 10, 20, 36, 65, 110, 185, 300,\ldots$ (starting with index $k= 0$) is the sequence A000712 of
the OEIS \cite{oeis}.
 This sequence 
 has many interesting interpretations and nice properties (see \cite{ch,oeis}), among which are the following two identities:
    $$\sum_{k=0}^\infty c(k)x^k=\prod_{i=1}^\infty\frac{1}{(1-x^i)^2},$$ and  $$c(k) =\sum_{i=0}^k p(i)p(k-i),$$
where $p(i)$ is the number of partitions of $i$.
Here is yet another interpretation of A000712: $c(k)$ is equal to
 the number of non-isomorphic basic graphs of $k$-permutations, or in other words, the number of all multi-graphs with exactly $k$-edges and with vertex degrees 1 or 2.
\end{rem}

The cycle type  of a permutation of  $n$ is shown as  $1^{a_1}2^{a_2}\ldots n^{a_n}$ where the superscripts $a_i\ge0$ indicate multiplicities.
It is known that the number of all permutations of $[n]$ of cycle type $1^{a_1}2^{a_2}\ldots n^{a_n}$ is equal to
$$\frac{n!}{a_1!a_2!2^{a_2}\cdots a_n!n^{a_n}}.$$
Here we determine the number of $k$-permutations sharing the same cycle structure.
For integers $n\ge\ell\ge0$ we  use  the {\em falling factorial} notation
$$(n)_\ell:=n(n-1)\cdots(n-\ell+1),$$
with the convention that $(n)_\ell=1$ if $\ell=0$.

\begin{thm} The number of  $k$-permutations of $[n]$ of the  cycle type
\begin{equation}\label{type}
1^{a_1}2^{a_2}\ldots k^{a_k}1'^{b_1}2'^{b_2}\ldots k'^{b_k}
\end{equation}
is equal to
$$\frac{k!}{\prod_{i=1}^k (i^{a_i}a_i!b_i!)}(n-k)_s,$$
where $s=b_1+\cdots+b_k$.
\end{thm}
\begin{proof}{
Note that $s=|\im(\pi)\setminus[k]|$
for any $k$-permutation $\pi$ of cycle type (\ref{type}).
Any $k$-permutation with cycle type  (\ref{type}) can be obtained from an arbitrary  permutation $\sigma$ of $[k]$ by first inserting
$a_1+\cdots+a_k$ pairs of parentheses to make the cycles and
inserting $s$ pairs consisting of a parenthesis and a bracket and then putting $s$ integers from $[n]\setminus[k]$ alongside the brackets to make the paths.
We count the number of ways that this can be done.
There are $k!$ ways to fill in the permutation $\sigma$. The integers inserted alongside the $s$ brackets can be regarded as an $s$-permutation of $[n]\setminus[k]$ and can be chosen in $(n-k)_s$ ways. These give the number $k!(n-k)_s$; but
we must correct our overcounting.
 Each of the $a_i$ cycles of length $i$ can be rotated around $i$ ways and be the same cycle, so we should divide
by $i^{a_i}$ for $i = 1, \ldots , k$. (Note that this is not the case for the paths; because in any path $(u_1\ldots u_\ell\,v]$ the last element $v$ always comes from $[n]\setminus[k]$ and rotating around the $u_i$'s yields different paths.)
 There are $a_i$  cycles of length $i$ and $b_i$ paths of length $i$ which can be permuted around in $a_i!$ and $b_i!$ ways, respectively, so we divide by
$a_i!b_i!$ for $i = 1, \ldots , k$.
}\end{proof}

\section{The cycle-type partition}

We partition $V(n,k)$ according to the cycle type of $k$-permutations. So $V(n,k)$ is partitioned into $c(k)$ cells (cf. Remark~\ref{c(k)}) where the $k$-permutations of each cell/part share the same cycle type.
Equivalently, two $k$-permutations belong to the same cell if they have isomorphic basic graphs.
We call this partition the {\em cycle-type partition} of $V(n,k)$. For instance,
 the cycle-type partition of 3-permutations of $[n]$ together with the corresponding cycle type of each cell are demonstrated in Table~\ref{cell3}.
\begin{table}
  \centering
  \begin{tabular}{cl}
\hline
   Type & Cell\\
   \hline &\vspace{-.3cm}\\
$ 111$& $V_1=\{(1,2,3)\}$\vspace{.1cm}\\
$ 12$ & $V_2=\{(1,3,2),(2,1,3),(3,2,1)\}$\vspace{.1cm}\\
$3$ & $V_3=\{(2,3,1),(3,1,2)\}$\vspace{.1cm}\\
$111'$& $V_4=\{(1,2,i),(1,i,3),(i,2,3)\mid 4\le i\le n\}$\vspace{.1cm}\\
$ 21'$ &  $V_5=\{(2,1,i),(3,i,1),(i,3,2)\mid 4\le i\le n\}$\vspace{.1cm}\\
$11'1'$ & $V_6=\{(1,i,j),(i,2,j),(i,j,3)\mid 4\le i,j\le n,i\ne j\}$\vspace{.1cm}\\
$12'$& $V_7=\{(1,3,i),(3,2,i),(1,i,2),(2,i,3),(i,2,1),(i,1,3)\mid 4\le i\le n\}$\vspace{.1cm}\\
$1'1'1'$& $V_8=\{(i,j,k)\mid 4\le i,j,k\le n,i\ne j\ne k\ne i\}$\vspace{.1cm}\\
$1'2'$ & $V_9=\{(2,i,j),(3,i,j),(i,1,j),(i,3,j),(i,j,1),(i,j,2)\mid 4\le i,j\le n,i\ne j\}$\vspace{.1cm}\\
$3'$& $V_{10}=\{(2,3,i),(3,1,i),(3,i,2),(2,i,1),(i,3,1), (i,1,2)\mid 4\le i\le n\}$\\
  \hline
\end{tabular}
  \caption{ The cycle-type partition of 3-permutations}\label{cell3}
\end{table}

An {\em equitable partition} of a graph $G$ is a partition $\Pi=(V_1,\ldots,V_m)$ of the vertex set
such that each vertex in $V_i$ has the same number $q_{ij}$ of neighbors in  $V_j$
for any $i,j$ (and possibly $i=j$). The {\em quotient matrix} of $\Pi$ is the $m\times m$ matrix $Q=(q_{ij})$.
It is well-known that every eigenvalue of the quotient matrix $Q$ is an eigenvalue of $G$ (see \cite[p. 24]{bh}). Under certain circumstances, the converse of this is also true as the next lemma shows.

We recall that a graph $G$ is {\em walk-regular} if for every positive integer $r$, the number of closed walks of length $r$ starting at a vertex $v$ is independent of the choice of $v$.
Clearly, vertex-transitive graphs are walk-regular and so are the arrangement graphs $A(n,k)$.

\begin{lem}[\cite{gm}]\label{god}  Let $G$ be a walk-regular graph with $\nu$ vertices. Let $\Pi=(V_1,\ldots,V_m)$ be an equitable partition of $G$ with $|V_1|=1$ and let
$Q$ be the quotient matrix of $\Pi$.
\begin{itemize}
\item[\rm(i)] Every eigenvalue of $G$ is an eigenvalue of $Q$.
\item[\rm(ii)] Let $S=\textnormal{diag}(\sqrt{\vert V_1\vert}, \sqrt{\vert V_2\vert},\dots, \sqrt{\vert V_m\vert})$ and $P=SQS^{-1}$. If $\{\x_1, \dots, \x_\ell\}$ is a full set of orthonormal eigenvectors of $P$ for the eigenvalue $\lambda$, then the multiplicity of $\lambda$ as an eigenvalue of $G$ is
$$
\nu \sum_{i=1}^\ell (\x_i)_1^2,
$$
where $(\x_i)_1$ denotes the first coordinate of $\x_i$.
\end{itemize}
\end{lem}

In the next theorem we show that the cycle-type partition is indeed an equitable partition for the arrangement graphs.
This partition is fine enough to meet the condition of Lemma~\ref{god} as the cell of type $1^k$ contains a single $k$-permutation, namely $(1,\ldots,k)$.

\begin{thm}\label{part} The cycle-type partition of $V(n,k)$ is an equitable partition of $A(n,k)$.
\end{thm}
\begin{proof}{Let $X$ and $Y$ (possibly $X=Y$) be two cells of the cycle-type partition of $V(n,k)$. Let $\pi,\rho\in X$. We show that there is a 
one-to-one correspondence between $N_Y(\pi)$ (the set of neighbors of $\pi$ in $Y$) and $N_Y(\rho)$.
Let
\begin{align*}
\pi & =\alpha_1\cdots \alpha_\ell\beta_1\cdots \beta_m,\\
\rho & =\gamma_1\cdots \gamma_\ell\delta_1\cdots \delta_m,
\end{align*}
be the cycle decompositions where $\alpha_i$ and $\gamma_i$ are cycles of the same length and $\beta_i$ and $\delta_i$ are paths of the same length. Let $\alpha_i=(a_{i1}\ldots a_{ir_i})$, $\gamma_i=(a_{i1}'\ldots a_{ir_i}')$,
$\beta_i=(b_{i1}\ldots b_{i(s_i-1)}\,b_{is_i}]$,  and $\delta_i=(b_{i1}'\ldots b_{i(s_i-1)}'\,b_{is_i}']$.
Note that $b_{is_i},b_{is_i}'\in [n]\setminus [k]$ and $a_{ij}, a_{ij}', b_{ie}, b_{ie}'\in [k]$ for all $i,j$ and $1\leq e\leq s_i-1$.

Let $\sigma_1$ be the permutation on $[k]$ that maps $a_{ij}\to a_{ij}'$ and $b_{ie}\to b_{ie}'$ for all $i,j$ and $1\leq e\leq s_i-1$ and $\sigma_2:[n]\setminus\left([k]\cup\im(\pi)\right)\to[n]\setminus\left([k]\cup\im(\rho)\right)$ be an arbitrary bijection.
Let
\begin{equation}
\sigma=\sigma_1\sigma_2\prod_{i=1}^m (b_{is_i}\ \ b_{is_i}').\notag
\end{equation}
Now, $\sigma$ is a permutation on $[n]$ and $\sigma \pi \sigma^{-1}=\rho$. 
Let $\zeta\in N_Y(\pi)$.
 Clearly, $\sigma \zeta \sigma^{-1}\in Y$.  We have $\zeta(u_0)\ne\pi(u_0)$ for exactly one $u_0\in [k]$ and $\zeta(u)=\pi(u)$ for all $u\in [k]\setminus \{u_0\}$. Now, $\sigma \pi \sigma^{-1}$ and  $\sigma \zeta \sigma^{-1}$ coincides at $u$ for all $u\in [k]\setminus \{u_0\}$ and differ at $u_0$. Hence, $\sigma \zeta \sigma^{-1}\in N_Y(\rho)$. The map
$\psi : N_Y(\pi)\to N_Y(\rho)$
defined by $\psi(\zeta)=\sigma \zeta\sigma^{-1}$ 
is a bijection, so the proof follows.
}\end{proof}

Nest, we  explicitly give the quotient matrix of the cycle-type partition of $A(n,k)$.
To state this result, we need more notations.
 We may denote the cycle
type of a $k$-permutation by the notation $[A,B]$, where $A$ and $B$ represent multisets of unprimed and primed integers, respectively. Thus,
elements of $A$ represent cycles while elements of $B$ represent paths.
 For a list $L$ and $i\in L$, we denote the resulting list of removing (one) $i$ from $L$ by $L_i$ and the resulting list of adding $r$ to $L$ by $L^r$.
We may combine these, so, e.g. $L_{i,j}^{r,t}$ is the list obtained from $L$ by removing $i,j$ and adding $r,t$.
The number of elements of $L$ counting multiplicities is denoted $|L|$.

In the next theorem, given a representative $\pi$ of the cycle type $[A,B]$, we determine the cycle types of the neighbors of $\pi$ 
which are described in the mutually exclusive subcases of the theorem. 
We also count the number of neighbors of $\pi$ of each type.

\begin{thm}\label{qutiont} Suppose that $\pi\in V(n,k)$ has the cycle type $[A,B]$.
Then the neighbors of $\pi$ are as follows.
\begin{itemize}
    \item[\rm(i)] For any $i\in A$ with multiplicity $a$,
\begin{itemize}
  \item[\rm(i.1)] $\pi$ has $ia(n-k-|B|)$ neighbors in $[A_i,B^i]$;
  \item[\rm(i.2)] for any $j\in B$ with multiplicity $b$, $\pi$ has $abi$ neighbors in $[A_i,B_j^{i+j}]$.
\end{itemize}
  \item[\rm(ii)] For any $j\in B$ with multiplicity $b$ and for any $\ell$ with $1\le\ell\le j$,
  \begin{itemize}
    \item[\rm(ii.1)] $\pi$ has $b$ neighbors in $[A^\ell,B_j^{j-\ell}]$;
    \item[\rm(ii.2)]  for any $m\in B_j$ with multiplicity $c$ and $m+\ell\ne j$, in $[A,B_{j,m}^{m+\ell,j-\ell}]$, $\pi$ has $2bc$ neighbors if
$m\ne j$ and $m-j+\ell\ge1$ and has $bc$ neighbors otherwise.
    \item[\rm(ii.3)] in $[A,B_j^{\ell,j-\ell}]$, $\pi$ has $2b(n-k-|B|)$ neighbors if $j\ne2\ell$ and  $b(n-k-|B|)$ neighbors if $j=2\ell$.
\end{itemize}
\item[\rm(iii)] If $B=j_1^{b_1}\ldots j_h^{b_h}$, then $\pi$ has
  $|B|(n-k-|B|)+\sum_{1\leqslant r<t\leqslant h}b_rb_t$  neighbors in $[A,B]$ (in particular, if $B=\emptyset$, then $\pi$ has no neighbor in $[A,B]$).
\end{itemize}
\end{thm}

\begin{proof}{ Let $G=\bg(\pi)$.
The vertex $\pi'\in V(n,k)$ is adjacent to $\pi$ if there is a unique $u\in[k]$ such that $\pi(u)\ne\pi'(u)$.
Considering the basic graphs, that means one can obtain $\bg(\pi')$ from $G$ by changing the arc $(u,\pi(u))$ to $(u,\pi'(u))$.
It follows that we can identify all the neighbors of $\pi$ in $A(n,k)$ by determining all basic graphs obtained from $G$ by changing an arc $(u,v)$ to $(u,w)$. Notice that we are allowed to do this only by changing the {\em head} of the arc and not the {\em tail} of the arc. Also as the resulting graph must be a basic graph, we have necessarily either $w\in[k]$ and it is a degree 1 vertex of $G$ or  $w\in[n]\setminus V(G)$.

We describe all the basic graphs obtained from $G$ through the above procedure in what follows.

(i) Let $C$ be a cycle of length $i$ in $G$. 
The only way to obtain a basic graph from $G$ by changing an arc $(u,v)$ of $C$ to $(u,w)$
is if either (a) $w\in[n]\setminus V(G)$  or (b) $w\in[k]$ and $w$ is the tail of a path of length $j$, say.
For the case (a), the resulting graph is isomorphic to the one obtained from $G$ by replacing $C$ with a path of length $i$, and thus it is isomorphic to the basic graph $H$ of $[A_i,B^i]$.
We have $(n-k-|B|)$ different choices for $w$ and also
 we may choose any of $ai$ edges of the $i$-cycles to obtain a basic graph isomorphic to $H$. Thus $\pi$ has  $ai(n-k-|B|)$ neighbors in $[A_i,B^i]$, proving (i.1). For the case (b),
the resulting graph is isomorphic to the basic graph $H$ of $[A_i,B_j^{i+j}]$. If there exist $b$ paths of length $j$ in $G$, then we have $abi$ different ways to obtain a basic graph isomorphic to $H$. This proves (i.2).

(ii) Let $(u_1\ldots u_j\,u_{j+1}]$ be a path of length $j$ in $G$.  
Consider the basic graph $H$ obtained from $G$
by changing the arc $(u_\ell,u_{\ell+1})$, $1\le\ell\le j$, to $(u,w)$. There are three possibilities, namely:
(a) $w=u_1$, (b) $w\in[k]\setminus\{u_1\}$, or (c) $w\in[n]\setminus V(G)$.
 In the case (a), $H$ is isomorphic to the basic graph of $[A^\ell,B_j^{j-\ell}]$. From any path of length $j$ we may obtain exactly one graph isomorphic to $H$. It follows that $\pi$ has $b$ neighbors in $[A^\ell,B_j^{j-\ell}]$, proving (ii.1).
In the case (b),  $w$ is necessarily the tail of a path of length $m$, say.
Thus $H$ is isomorphic to the basic graph of $[A,B_{j,m}^{m+\ell,j-\ell}]$. This gives $bc$ neighbors for $\pi$ in $[A,B_{j,m}^{m+\ell,j-\ell}]$. 
Note that if $m\ne j$ and $r:=m-j+\ell\ge1$,  then we may obtain more neighbors for  $\pi$ in $[A,B_{j,m}^{m+\ell,j-\ell}]$. 
As $m\ne j$ and $m\in B_j$ we have $j\in B_m$. With the same reasoning as above, $\pi$ has 
$bc$ neighbors in $[A,B_{m,j}^{j+r,m-r}]$. However, $B_{m,j}^{j+r,m-r}=B_{j,m}^{m+\ell,j-\ell}$.  
Therefore, if $m\ne j$ and $m-j+\ell\ge1$, the total number of neighbors of $\pi$ in $[A,B_{j,m}^{m+\ell,j-\ell}]$ is $2bc$.
This proves (ii.2).
In the case (c), we have $H$ isomorphic to the basic graph of $[A,B_j^{\ell,j-\ell}]$. 
We have $b$ paths of length $j$ and also $(n-k-|B|)$ different choices for $w$. This gives $b(n-k-|B|)$ neighbors for $\pi$ in $[A,B_j^{\ell,j-\ell}]$. But, if $j\ne2\ell$, then by interchanging  $\ell$ and $j-\ell$ we obtain another $b(n-k-|B|)$ copies of $H$ from $G$.  This proves (ii.3).

(iii) There are only two ways to change an arc $(u,v)$ to $(u,w)$ in $G$ and obtain a basic graph isomorphic to $G$, namely
(a)  by choosing $v$ to be the head of some path in $G$
and $w\in[n]\setminus V(G)$, (b) by the same way as in (ii.2) but with $m+\ell=j$. For (a), $v$ can be any of the heads of the $|B|$ paths of $G$ and $w$ any of the $(n-k-|B|)$ elements of $[n]$ not appearing in $G$. This gives $|B|(n-k-|B|)$ different copies of $G$. Now we count the rest of the neighbors of $\pi$ coming from (b).
For any $m,j\in B$ with $m<j$, letting $\ell=j-m$, we have $B_{j,m}^{m+\ell,j-\ell}=B$.
There are $\sum_{1\leqslant r<t\leqslant h}b_rb_t$ different ways to choose pairs of $m,j\in B$ with $m<j$.
The proof now follows by summing up the number of neighbors given in the cases (a) and (b).
}\end{proof}

\section{Eigenvalues of $A(n,k)$}

 In this section we determine the eigenvalues of $A(n,k)$ for $k\le7$.
As mentioned before, the arrangement graphs are vertex-transitive and thus walk-regular.
Furthermore, the cycle-type partition of $A(n,k)$ contains a cell of cardinality $1$.
So Lemma~\ref{god} can be applied to derive the eigenvalues of of $A(n,k)$.
For the graphs $A(n,2)$, though, we employ a different method taking into account a characterization which directly gives the eigenvalues.

Since $A(n,k)$ is a $k(n-k)$-regular graph, the largest eigenvalue is $k(n-k)$. When $n=k$, the edge set is an empty set. So,  $A(k,k)$ has one eigenvalue only, which is 0.
So we shall assume that $n>k$.
 When $k=1$, $A(n,1)$ is the complete graph with $n$ vertices. Therefore the eigenvalues of $A(n,1)$ are $(n-1)$ with multiplicity 1, and $-1$ with multiplicity $n-1$.

\subsection{Eigenvalues of $A(n,2)$}

 We use  the following well-known lemma  (see \cite[p. 10]{bh}) to derive the eigenvalues of $A(n,2)$.

\begin{lem}\label{line} If $G$ is an $r$-regular graph $(r\ge 2)$ with $\nu$ vertices and $\varepsilon$ edges,
and eigenvalues $\la_i,$ $i=1,\ldots,\nu$, then the line graph of $G$ is $(2r-2)$-regular with eigenvalues
$\la_i+r-2$, $i=1,\ldots,\nu$ together with $\varepsilon-\nu$ times $-2$.
\end{lem}

\begin{pro} Let $n\geq 3$. The eigenvalues of the arrangement graph $A(n,2)$ are
$$(-2)^{[n^2-3n+1]},~ (n-4)^{[n-1]},~ (n-2)^{[n-1]},~ (2n-4)^{[1]}, $$
where the superscripts indicate multiplicities.
\end{pro}
\begin{proof}{Let $H_n$ denote the graph obtained from the complete bipartite graph $K_{n,n}$ after removing a perfect matching.
Assume that $\{a_1,\ldots,a_n,b_1,\ldots,b_n\}$ is the vertex set and $\{a_ib_j\mid1\le i,j\le n,i\ne j\}$ is the edge set of $H_n$.
In the line graph $L(H_n)$ of $H_n$, two distinct edges $a_ib_j$ and $a_rb_s$ are adjacent if and only if either $i=r$ or $s=j$.
This shows that the map which sends the vertex $(i,j)$ of $A(n,2)$ to the vertex $a_ib_j$ of $L(H_n)$ defines an isomorphism between $A(n,2)$ and $L(H_n)$.

Now, it suffices to determine the eigenvalues of $L(H_n)$. Since the adjacency matrix of $H_n$ is $(J_2-I_2)\otimes(J_n-I_n)$, with  `$\otimes$' denoting the Kronecker product, the eigenvalues of $H_n$ are $\pm(n-1)$ with multiplicity $1$ and $\pm1$ with multiplicity $n-1$. The result now follows by applying Lemma~\ref{line}.
}\end{proof}

\subsection{Eigenvalues of $A(n,3)$ and $A(n,4)$}

Throughout this subsection, we use the notation $n_i$ for $(n-i)$ for saving space in tables and arrays. 

\begin{thm}
For $n\ge4$, the eigenvalues of $A(n,3)$ are
\begin{align*}&(-3)^{[n(n-2)(n-4)-1]},~ (n-7)^{[n(n-3)/2]},~ (n-6)^{[(n-2)(n-1)]},~ (n-4)^{[n(n-3)]},~ \\ &(n-3)^{[(n-1)(n-2)/2]},~ (2n-9)^{[n-1]},~ (2n-6)^{[2(n-1)]},~ (3n-9)^{\,[1]}.
\end{align*}
\end{thm}
\begin{proof}{The eigenvalues of $A(4,3)$ and $A(5,3)$ are determined by a computer; these are
$$\{-3^{[1]},\, -2^{[6]},\, -1^{[3]},\, 0^{[4]},\, 1^{[3]},\, 2^{[6]},\, 3^{[1]}\}~~\hbox{and}~~
\{-3^{[14]},\, -2^{[5]},\, -1^{[12]},\, 1^{[14]},\, 2^{[6]},\, 4^{[8]},\, 6^{[1]}\},$$
respectively, which agree with the assertion. (Note that letting $n=4$ in the assertion, the sum of the multiplicities of $-3$ and $n-7$ equals $1$.) Now, we may assume that $n\ge6$.

By Theorem~\ref{qutiont}, the quotient matrix of the cycle-type partition of $A(n,3)$ is the following where the cells are indexed as in Table~\ref{cell3}:
$$Q={\small\left(
  \begin{array}{cccccccccc}
0& 0& 0& 3n_3& 0& 0& 0& 0& 0& 0\\ 0& 0& 0& 0& n_3& 0& 2n_3& 0& 0& 0\\ 0& 0& 0& 0& 0& 0& 0& 0& 0& 3n_3\\ 1& 0& 0& n_4& 0& 2n_4& 2& 0& 0& 0\\ 0& 1& 0& 0& n_4& 0& 0& 0& 2n_4& 2\\ 0& 0& 0& 2& 0& 2n_5& 2& n_5& 2& 0\\ 0& 1& 0& 1& 0& n_4& n_4& 0& n_4& 1\\ 0& 0& 0& 0& 0& 3& 0& 3n_6& 6& 0\\ 0& 0& 0& 0& 1& 1& 1& n_5& 2n-9& 2\\ 0& 0& 1& 0& 1& 0& 1& 0& 2n_4& n_4\\
\end{array}
\right).}$$
By computation, we worked out the eigenvalues and eigenvectors of $Q$ as given in Table~\ref{eig3}.

\begin{table}[h]
  \centering
\begin{tabular}{cc}
\hline
   Eigenvalue & Eigenvector\\
   \hline &\vspace{-.15cm}\\
$  -3$&$ \left[ -3n_3n_4,\, n_3n_4,\, 0,\, 3n_4,\, -n_4,\, -2,\, -n_4,\, 0,\, 1,\, 0\right]$\vspace{.15cm}\\
$   -3$&$ \left[ -n_3n_4n_5,\, 0,\, 0,\, n_4n_5,\, 0,\, -n_5,\, 0,\, 1,\, 0,\, 0\right]$\vspace{.15cm}\\
 $   -3$&$ \left[ -n_3,\, n_3,\, -n_3,\, 1,\, -1,\, 0,\, -1,\, 0,\, 0,\, 1\right]$\vspace{.15cm}\\
$ n-7$&$ \left[ 3n_3n_4,\, 3n_3n_4,\, 3n_3n_4,\, n_4n_7,\, n_4n_7,\, 22-4n,\, n_4n_7,\, 18,\, 22-4n,\, n_4n_7\right]$\vspace{.15cm}\\
  $  n-6$&$ \left[6n_3,\, 0,\, -3n_3,\, 2n_6,\, -2n_3,\, -6,\, n_3,\, 0,\, 3,\, -n_6\right]$\vspace{.15cm}\\
    $ n-4$&$ \left[ -6n_3,\, 0,\, 3n_3,\, -2n_4,\, -2n_1,\, 2,\, n_1,\, 0,\, -1,\, n_4\right]$\vspace{.15cm}\\
    $ n-3$&$ \left[ 3,\, -3,\, 3,\, 1,\, -1,\, 0,\, -1,\, 0,\, 0,\, 1\right]$\vspace{.15cm}\\
    $ 2n-9$&$ \left[ 3n_3,\, 3n_3,\, 3n_3,\, 2n-9,\, 2n-9,\, n_9,\, 2n-9,\, -9,\, n_9,\, 2n-9\right]$\vspace{.15cm}\\
$ 2n-6$&$ \left[ -6,\, 0,\, 3,\, -4,\, 2,\, -2,\, -1,\, 0,\, 1,\, 2\right]$\vspace{.15cm}\\
$ 3n-9$&$ \left[ 1,\, 1,\, 1,\, 1,\, 1,\, 1,\, 1,\, 1,\, 1,\, 1\right]$\\ \hline
\end{tabular}
 \caption{The eigenvalues and the (transposed) eigenvectors of the quotient matrix of $A(n,3)$}\label{eig3}
\end{table}

Let
\begin{align*}
S={\rm diag} &\left(\sqrt1, \sqrt3, \sqrt2, \sqrt{3(n-3)}, \sqrt{3(n-3)}, \sqrt{3(n-3)(n-4)}, \sqrt{6(n-3)},\right.\\ &\left.~~\sqrt{(n-3)(n-4)(n-5)},\sqrt{6(n-3)(n-4)}, \sqrt{6n-18}\right),
\end{align*}
and $P=SQS^{-1}$. Note that $\bf v$ is an eigenvector of $Q$ for the eigenvalue $\lambda$ if and only if $S\bf v$ is an eigenvector of $P$ for the eigenvalue $\lambda$. For any eigenvalue $\la$ of $Q$ with multiplicity $1$ and with eigenvector $\bf v$, setting $\w=S\bf v$,  the multiplicity of $\la$ as an eigenvalue of $A(n,3)$ is obtained by Lemma~\ref{god} as
$$n(n-1)(n-2)\frac{(\w)_1^2}{\w^\top\w}.$$
For instance,
 ${\bf v}=\left[ 3(n-3), 3(n-3), 3(n-3), 2n-9, 2n-9, n-9, 2n-9, -9, n-9, 2n-9\right]^\top$ is an eigenvector of $Q$ for the 1-fold eigenvalue $2n-9$. Hence,
{\small\begin{align*}
 {\bf w}=S{\bf v}= &\left[ 3(n-3), 3\sqrt3(n-3), 3\sqrt2(n-3), (2n-9)\sqrt{3(n-3)}, (2n-9)\sqrt{3(n-3)}, (n-9)\sqrt{3(n-3)(n-4)},\right. \\
& ~~ \left. (2n-9)\sqrt{6(n-3)}, -9\sqrt{(n-3)(n-4)(n-5)}, (n-9)\sqrt{6(n-3)(n-4)}, (2n-9)\sqrt{6n-18}\right]^\top,
\end{align*}}
 is an eigenvector of $P$ for the eigenvalue $2n-9$. The multiplicity of $2n-9$ as an eigenvalue of $A(n,3)$ is then equal to
$$n(n-1)(n-2)\frac{(\w)_1^2}{\bf w^\top\bf w}=n(n-1)(n-2)\frac{9(n-3)^2}{9n(n-2)(n-3)^2}=n-1.$$
By similar calculations, the multiplicities of other eigenvalues of  $A(n,3)$ can be determined (except for the eigenvalue $-3$).  The multiplicity of $-3$ is $n(n-1)(n-2)$ minus the sum of all the multiplicities of the rest of the eigenvalues.
}\end{proof}

\begin{thm}
For $n\ge5$, the eigenvalues of $A(n,4)$ are as follows:
$$\begin{array}{lll}
(-4)^{[n(n-3)(n^2-7n+8)+1]}& (n-10)^{[n(n-1)(n-5)/6]}& (n-9)^{[n(n-2)(n-4)]}\\
(n-8)^{[(n-1)(n-2)(n-3)/2]}& (n-7)^{[2n(n-2)(n-4)/3]}& (n-6)^{[n(n-1)(n-5)/2]}\\
(n-5)^{[n(n-2)(n-4)]}& (n-4)^{[(n-1)(n-2)(n-3)/6]}& (2n-14)^{[n(n-3)/2]}\\
(2n-12)^{[3(n-1)(n-2)/2]}& (2n-10)^{[3n(n-3)/2]}& (2n-8)^{[5n(n-3)/2+3]}\\
(3n-16)^{[n-1]}& (3n-12)^{[3(n-1)]}& (4n-16)^{[1]}.
\end{array}$$
\end{thm}

\begin{proof}{The eigenvalues of $A(5,4)$, $A(6,4)$ and $A(7,4)$ are determined by a computer; these are
\begin{align*}
&\{-4^{[1]},\, -3^{[12]},\, -2^{[28]},\, -1^{[4]},\, 0^{[30]},\, 1^{[4]},\, 2^{[28]},\, 3^{[12]},\, 4^{[1]}\}, \\
&\{-4^{[42]},\, -3^{[48]},\, -2^{[39]},\, -1^{[32]},\, 0^{[45]},\, 1^{[48]},\, 2^{[42]},\, 4^{[48]},\, 6^{[15]},\, 8^{[1]}\},\ \textnormal{and} \\
&\{-4^{[225]},\, -3^{[14]},\, -2^{[105]},\, -1^{[60]},\, 0^{[84]},\, 1^{[42]},\, 2^{[150]},\, 3^{[20]},\, 4^{[42]},\, 5^{[6]}\, 6^{[73]},\, 9^{[18]},\, 12^{[1]}\},
\end{align*}
respectively, which agree with the assertion. So we may assume that $n\ge6$.

We consider the following order for the cycle types of 4-permutations:
$$1111, 112, 22, 13, 4, 1111', 121', 31',  111'1', 21'1', 112', 22', 11'1'1',  11'2', 13', 1'1'1'1', 1'1'2', 2'2', 1'3', 4'.$$
We  order the cells of the cycle-type partition according to the  above ordering.
From Theorem~\ref{qutiont}, the quotient matrix $Q$ of the cycle-type partition of $A(n,4)$ is computed as:
{\small$$\left(\begin{array}{cccccccccccccccccccc}
0& 0& 0& 0& 0& 4n_4& 0& 0& 0& 0& 0& 0& 0& 0& 0& 0& 0& 0& 0& 0\\ 0& 0& 0& 0& 0& 0& 2n_4& 0& 0& 0& 2n_4& 0& 0& 0& 0& 0& 0& 0& 0& 0\\ 0& 0& 0& 0& 0& 0& 0& 0& 0& 0& 0& 4n_4& 0& 0& 0& 0& 0& 0& 0& 0\\ 0& 0& 0& 0& 0& 0& 0& n_4& 0& 0& 0& 0& 0& 0& 3n_4& 0& 0& 0& 0& 0\\
 0& 0& 0& 0& 0& 0& 0& 0& 0& 0& 0& 0& 0& 0& 0& 0& 0& 0& 0& 4n_4\\ 1& 0& 0& 0& 0& n_5& 0& 0& 3n_5& 0& 3& 0& 0& 0& 0& 0& 0& 0& 0& 0\\ 0& 1& 0& 0& 0& 0& n_5& 0& 0& n_5& 0& 1& 0& 2n_5& 2& 0& 0& 0& 0& 0\\0& 0& 0& 1& 0& 0& 0& n_5& 0& 0& 0& 0& 0& 0& 0& 0& 0& 0& 3n_5& 3\\ 0& 0& 0& 0& 0& 2& 0& 0& 2n_6& 0& 2& 0& 2n_6& 4& 0& 0& 0& 0& 0& 0\\ 0& 0& 0& 0& 0& 0& 2& 0& 0& 2n_6& 0& 2& 0& 0& 0& 0& 2n_6& 0& 4&0\\ 0& 1& 0& 0& 0& 1& 0& 0& n_5& 0& n_5& 0& 0& 2n_5& 2& 0& 0& 0& 0& 0\\ 0& 0& 1& 0& 0& 0& 1& 0& 0& n_5& 0& n_5& 0& 0& 0& 0& 0& 2n_5& 0& 2\\ 0& 0& 0& 0& 0& 0& 0& 0& 3& 0& 0& 0& 3n_7& 6& 0& n_7& 3& 0& 0& 0\\ 0& 0& 0& 0& 0& 0& 1& 0& 1& 0& 1& 0& n_6& 2n-11& 2& 0& n_6& 1& 1& 0\\ 0& 0& 0& 1& 0& 0& 1& 0& 0& 0& 1& 0& 0& 2n_5& n_5& 0& 0& 0& n_5& 1\\ 0& 0& 0& 0& 0& 0& 0& 0& 0& 0& 0& 0& 4& 0& 0& 4n_8& 12& 0& 0& 0\\ 0& 0& 0& 0& 0& 0& 0& 0& 0& 1& 0& 0& 1& 2& 0& n_7& 3n-19& 2& 4& 0\\ 0& 0& 0& 0& 0& 0& 0& 0& 0& 0& 0& 2& 0& 2& 0& 0& 2n_6& 2n_6& 2& 2\\ 0& 0& 0& 0& 0& 0& 0& 1& 0& 1& 0& 0& 0& 1& 1& 0& 2n_6& 1& 2n-11 & 2\\ 0& 0& 0& 0& 1& 0& 0& 1& 0& 0& 0& 1& 0& 0& 1& 0& 0& n_5& 2n_5& n_5\\
\end{array}\right).$$}

By computation, we obtain the eigenvalues and  eigenvectors of $Q$  as given in Table~\ref{eig4}.

\begin{table}[h]
  \centering
\begin{tabular}{rl}
\hline
   Eigenvalue & Eigenvector\\
   \hline &\vspace{-.18cm}\\
 $-4$&{\footnotesize$\left[5n_4n_5n_6 ,\, 0,\, -n_4n_5n_6,\, -n_4n_5n_6,\, n_4n_5n_6,\,
        -5 n_5n_6,\, 0,\, n_5n_6,\,5 n_6,\, -n_6,\, 0,\, n_5n_6 ,\, -3,\, -n_6,\, n_5n_6,\, 0,\, 1,\, 0,\, 0,\, -n_5n_6\right]$}\vspace{.18cm}\\
  $-4$&{\small$\left[2n_4n_5,\, 0,\, 0,\, -n_4n_5,\, n_4n_5,\,-2n_5,\, 0,\, n_5,\, 2,\, 0,\, 0,\, 0,\, 0,\, -1,\, n_5,\, 0,\, 0,\, 1,\, 0,\,-n_5\right]$}\vspace{.18cm}\\
  $-4$&{\small$\left[2n_4n_5,\, 0,\, -2n_4n_5,\, -n_4n_5,\, 2n_4n_5,\, -2n_5,\, 0,\, n_5,\, 2,\, -2,\, 0,\, 2n_5,\, 0,\, -1,\, n_5,\, 0,\, 0,\, 0,\, 1,\, -2n_5\right]$}\vspace{.18cm}\\
  $-4$&{\small$\left[n_4,\, -n_4,\, n_4,\, n_4,\, -n_4,\, -1,\, 1,\, -1,\, 0,\, 0,\, 1,\, -1,\, 0,\, 0,\, -1,\, 0,\, 0,\, 0,\, 0,\, 1\right]$}\vspace{.18cm}\\
  $-4$&{\small$\left[n_4n_5n_6n_7,\, 0,\, 0,\, 0,\, 0,\, -n_5n_6n_7,\, 0,\, 0,\, n_6n_7,\, 0,\, 0,\, 0,\,  -n_7,\, 0,\, 0,\, 1,\, 0,\, 0,\, 0,\, 0\right]$}\vspace{.18cm}\\
     $3 n-16$&{\small$\left[4n_4,\,4n_4,\,4n_4,\,4n_4,\,4n_4,\, 3n-16,\,  3n-16,\,  3n-16,\,-2n_8,\, -2n_8,\, 3n-16,\,  3n-16,\,n_{16}-2n_8,\,\right.$}\\
     &~~{\small$\left.3n-16,\, - 16,\, n - 16,\, -2n_8,\,-2n_8,\, 3n-16\right]$}\vspace{.18cm}\\
 $ n-10 $&{\small$\left [-4n_4n_5n_6,\, -4n_4n_5n_6,\, -4n_4n_5n_6,\, -4n_4n_5n_6,\, -4n_4n_5n_6,\, -n_5n_6n_{10},\, -n_5n_6n_{10},\, -n_5n_6n_{10},\, 4n_6n_8,\right.$}\\ &~~{\small$\left. 4n_6n_8,\,-n_5n_6n_{10},\, -n_5n_6n_{10},\, 132-18n,\, 4n_6n_8,\, -n_5n_6n_{10},\, 96,\, 132-18n,\, 4n_6n_8,\, 4n_6n_8,\, -n_5n_6n_{10}\right]$}\vspace{.18cm}\\
  $n-9$& {\small$\left[-12n_4n_5,\, -4n_4n_5,\, 4n_4n_5,\, 0,\, 4n_4n_5,\, -3n_5n_9,\, 4n_5,\, 3(n-5)^2/2,\, 10n-66,\, -2n_5,\, -2n_5n_7,\, \right.$}\\
    &~~{\small$\left.n_5n_9,\,-24,\, 3n-19,\, -(n-5)^2/2,\, 0,\, 8,\, -4n_7,\, 19-3n,\, n_5n_9\right]$}\vspace{.18cm}\\
  $n-8 $& {\small$\left[-12n_4,\, 4n_4,\, 4n_4,\, 0,\, -4n_4,\, -3n_8,\, 3n-16,\, -3n_4,\, 8,\, -8,\, -n,\, n_8,\, 0,\, -4,\, n_4,\, 0,\, 0,\, 0,\, 4,\, -n_8\right]$}\vspace{.18cm}\\
  $n - 7$& {\small$\left[4 n_4,\, 0,\, 4 n_4,\, -2n_4,\, 0,\, n_7,\, 0,\, - n_7/2,\, -2,\, -2,\, 0,\, n_7,\, 0,\, 1,\, -n_7/2,\, 0,\, 0,\, -2,\, 1,\, 0\right]$}\vspace{.18cm}\\
  $n - 6$&{\small$\left[-12n_4n_5,\, -4n_4n_5,\, 4n_4n_5,\, 0,\, 4n_4n_5,\, -3n_5n_6,\, -n_5(3n-10),\, -3n_2n_5,\, 4n_6,\, 4n_2,\, (n+2)n_5,\,\right.$}\\
       &~~{\small$\left.n_5n_6,\, -6,\, -4,\, n_2n_5,\, 0,\, 2,\, -4n_4,\, 4,\,n_5n_6\right]$}\vspace{.18cm}\\
  $n - 5$&{\small$\left [-12 n_4,\, 4 n_4,\, 4 n_4 ,\, 0,\, -4 n_4,\, -3 n_5,\, -4,\,n_3/2,\,2,\, -2,\, 2 n_3,\, n_5,\, 0,\, -1,\, - n_1/2,\, 0,\, 0,\,
       0,\, 1,\, -n_5\right]$}\vspace{.18cm}\\
  $n - 4$& {\small$\left[-4,\, 4,\, -4,\, -4,\, 4,\, -1,\, 1,\, -1,\, 0,\, 0,\, 1,\, -1,\, 0,\, 0,\, -1,\, 0,\, 0,\, 0,\,0,\, 1\right]$}\vspace{.18cm}\\
  $2 n - 14$&{\small$\left[-6n_4n_5,\, -6n_4n_5,\, -6n_4n_5,\, -6n_4n_5,\, -6n_4n_5,\, -3n_5n_7,\, -3n_5n_7,\, -3n_5n_7,\, \tau,\, \tau,\,-3n_5n_7,\,\right.$}\\
     &~~{\small$\left.  -3n_5n_7,\, 9n_3,\, \tau,\, -3n_5n_7,\, -72,\, 9n_3,\,\tau,\,\tau,\, -3n_5n_7\right]$ ~where $\tau=21n-n^2-92$}\vspace{.18cm}\\
  $2 n - 12$&{\footnotesize$\left[-6n_4, -2n_4,\, 2n_4, 0, 2n_4,\, 18-3n,\, 2, 3n_{12}/2, -n_{12},\, n_4, -2n_5,\, n_6,\, 6, -n_8/2, -n_4/2,\, 0, -2, -4, n_4/2,\, n-6\right]$}\vspace{.18cm}\\
  $2 n - 10$&{\small$\left [6n_4,\, 2n_4,\, -2n_4,\, 0,\, -2n_4,\, 3n_5,\, n_3,\, 3,\, n_8,\, n,\, n_7,\, -n_5,\, -3,\, -2,\, -1,\, 0,\, 1,\, -n_4,\, 2,\, -n_5\right]$}\vspace{.18cm}\\
  $2 n - 8$&{\small$\left[12,\, -2,\, 4,\, -3,\, 2,\, 6,\, -2,\, 0,\, 2,\, 0,\, 0,\, 2,\, 0,\, -1,\, -2,\, 0,\, 0,\,1,\, 0,\, 1\right]$}\vspace{.18cm}\\
  $2 n - 8$&{\small$\left [12,\, -4,\, -4,\, 0,\, 4,\, 6,\, -4,\, 3,\, 2,\, -2,\, 0,\, -2,\, 0,\, -1,\, -1,\, 0,\, 0,\,0,\, 1,\, 2\right]$}\vspace{.18cm}\\
  $3 n - 12$&{\small$\left[12,\, 4,\, -4,\, 0,\, -4,\, 9,\, 1,\, -3,\, 6,\, -2,\, 5,\, -3,\, 3,\, 2,\, 1,\, 0,\, -1,\,-2,\, -2,\, -3\right]$}\vspace{.18cm}\\
  $4 n - 16$&{\small$\left[1,\, 1,\, 1,\, 1,\, 1,\, 1,\, 1,\, 1,\, 1,\, 1,\, 1,\, 1,\, 1,\, 1,\, 1,\, 1,\, 1,\, 1,\, 1,\, 1\right]$}\vspace{.18cm}\\
\hline
\end{tabular}
 \caption{The eigenvalues and the (transposed) eigenvectors of the quotient matrix of $A(n,4)$}\label{eig4}
\end{table}

Let
{\small\begin{align*}
S={\rm diag} &\left(
\sqrt1, \sqrt6, \sqrt3, \sqrt8, \sqrt6, \sqrt{4(n-4)}, \sqrt{12(n-4)}, \sqrt{8(n-4)}, \sqrt{6(n-4)(n-5)}, \sqrt{6(n-4)(n-5)},\sqrt{12(n-4)}, \right.\\
&~~ \sqrt{12(n-4)}, \sqrt{4(n-4)(n-5)(n-6)}, \sqrt{24(n-4)(n-5)}, \sqrt{24(n-4)}, \sqrt{(n-4)(n-5)(n-6)(n-7)},\\
&~~\left. \sqrt{12(n-4)(n-5)(n-6)}, \sqrt{12(n-4)(n-5)},\sqrt{24(n-4)(n-5)}, \sqrt{24(n-4)}\right),
\end{align*}}
and $P=SQS^{-1}$. For any eigenvalue $\la$ of $Q$ with multiplicity $1$ and with eigenvector $\bf v$, setting $\w=S\bf v$,  the multiplicity of $\la$ as an eigenvalue of $A(n,4)$ is obtained by Lemma~\ref{god} as
$$n(n-1)(n-2)(n-3)\frac{(\w)_1^2}{\w^\top\w}.$$
However, if the multiplicity of $\la$ for $Q$ is larger than $1$, we need to find an orthogonal set of eigenvectors for $\la$ and $P$.
Besides $-4$, only $2n-8$ is such an eigenvalue.
Set
\begin{align*}
\w_1&= S\left[12, -2, 4, -3, 2, 6, -2, 0, 2, 0, 0, 2, 0, -1, -2, 0, 0,1, 0, 1\right]^\top,\\
\w_2&=S\left [12, -4, -4, 0, 4, 6, -4, 3, 2, -2, 0, -2, 0, -1, -1, 0, 0,0, 1, 2\right]^\top.
\end{align*}
Now the vectors ${\bf u}_1,{\bf u}_2$ with ${\bf u}_1=\w_1$ and ${\bf u}_2=\w_2-\frac{\w_2^\top\w_2}{\w_1^\top\w_2}\w_1$
forms a set of orthogonal eigenvectors of $P$ for $2n-8$.
By Lemma~\ref{god}, the multiplicity of $2n-8$ as an eigenvalue of $A(n,4)$ is
\begin{align*}
   &n(n-1)(n-2)(n-3)\left(\frac{(\bu_1)_1^2}{\bu_1^\top\bu_1}+\frac{(\bu_2)_1^2}{\bu_2^\top\bu_2}\right)\\
&~~\,=n(n-1)(n-2)(n-3)\left(\frac{12}{6-15n+5n^2}+\frac{(n^2-3n+6)^2}{2n(n-1)(n-2)(n-3)(5n^2-15n+6)}\right)\\
&~~\,=\frac{5n(n-3)}{2}+3.
\end{align*}
Now, the multiplicity of $-4$ is $n(n-1)(n-2)(n-3)$ minus the sum of all the multiplicities of the rest of the eigenvalues.
}\end{proof}

\subsection{Eigenvalues of $A(n,k)$ for $k=5,6,7$}

In a similar fashion as for $A(n,3)$ and $A(n,4)$, we are able to determine the complete set of eigenvalues of more families of the arrangement graphs.
 The eigenvalues of $A(n,k)$ for $k=5,6,7$ are given in Tables~\ref{eig5}, \ref{eig6} and \ref{eig7}.
We would like to point out that by using our method it is possible to compute the eigenvalues of the graphs $A(n,k)$ for some larger values of $k>7$.
\vspace{.3cm}
\begin{table}
\centering
{\small$\begin{array}{lll}
\hline \vspace{-.25cm} \\
(-5)^{[n^5-15n^4+75n^3-145n^2+89n-1]}& ( n-13)^{[n(n-1)(n-2)(n-7)/24]}& ( n-12)^{[n(n-1)(n-3)(n-6)/2]}\\
(n-11)^{[n(n-5)(7n)^2-35n+37)/6]}& (n-10)^{[(n-1)(n-2)(n-3)(n-4)/6]}& (n-9)^{[5n(n-3)(n^2-7n+8)/4]}\\
( n-8)^{[n(n-1)(n-2)(n-7)/6]}& (n-7)^{[n(n-1)(7n)^2-63n+131)/6]}& ( n-6)^{[n(n-2)(n-3)(n-5)/2]}\\
(n-5)^{[(n-1)(n-2)(n-3)(n-4)/24]}& ( 2n-19)^{[n(n-1)(n-5)/6]}& ( 2n-17)^{[4n(n-2)(n-4)/3]}\\
 (  2n-15)^{[(n-1)(n-2)(n-3)]}& ( 2n-14)^{[n(7n)^2-42n+50)/2]}& ( 2n-12)^{[2n(n-2)(n-4)]}\\
 ( 2n-11)^{[5n(n-1)(n-5)/6]}& ( 2n-10)^{[(n-2)(7n^2-28n+6)/3]}& (3n-23)^{[n(n-3)/2]}\\
 ( 3n-20)^{[2(n-1)(n-2)]}& ( 3n-18)^{[2n(n-3)]}& ( 3n-15)^{[(11n^2-33n+12)/2]}\\
 (  4n-25)^{[n-1]}& (  4n-20)^{[4n-4]}& ( 5n-25)^{[1]}
\\ \hline
\end{array}$}
\caption{The eigenvalues of $A(n,5)$}\label{eig5}
\end{table}

\begin{table}
\centering
\end{table}

\begin{table}
{\small$\begin{array}{lll}
\hline \vspace{-.25cm} \\
(-6)^{[n^6-21n^5+160n^4-545n^3+814n^2-415n+1]}&
(n-16)^{[n(n-1)(n-2)(n-3)(n-9)/120]}&
(n-15)^{[n(n-1)(n-2)(n-4)(n-8)/6]}\\
 (n-14)^{[n(n-1)(n-7)(7n^2-49n+78)/8]}&
 (n-13)^{[n(n-3)(n-6)(n^2-6n+6)]}&
 (n-12)^{[(n-1)(n-2)(n-5)(n^2-7n+2)/4]}\\
  (n-11)^{[n(n-4)(7n^3-77n^2+217n-162)/4]}&
 (n-10)^{[n(n-1)(n-3)(n-4)(n-7)/4]}&
  (n-9)^{[n(n-1)(n-2)(n^2-12n+34)]}\\
 (n-8)^{[ n(n-1)(n-3)(7n^2-77n+202)/8]}&
 (n-7)^{[n(n-2)(n-3)(n-4)(n-6)/6]}&
  (n-6)^{[(n-1)(n-2)(n-3)(n-4)(n-5)/120]}\\
  (2n-24)^{[n(n-1)(n-2)(n-7)/24]}&
  (2n-22)^{[ 5n(n-1)(n-3)(n-6)/8]}&
   (2n-20)^{[n(n-5)(2n-3)(2n-7)/2]}\\
 (2n-18)^{[(7n^3-63n^2+136n-40)(n-1)/4]}&
 (2n-17)^{[2n(n-2)(n-3)(n-5)]}&
  (2n-16)^{[15n(n-1)(n-3)(n-6)/8]}\\
  (2n-15)^{[4n(n-1)(n-4)(n-5)/3]}&
 (2n-14)^{[n(n-2)(13n^2-104n+171)/8]}&
  (2n-13)^{[2n(n-1)(n-3)(n-6)]}\\
 (2n-12)^{[(7n^4-70n^3+217n^2-210n+20)/4]}&
   (3n-30)^{[n(n-1)(n-5)/6]}&
 (3n-27)^{[5n(n-4)(n-2)/3]}\\
 (3n-23)^{[3n(n-2)(n-2)]}&
 (3n-24)^{[5(n-4)(n-1)^2/2]}&
 (3n-21)^{[10n(n-2)(n-4)/3]}\\
  (3n-20)^{[3n(n-1)(n-5)/2]}&
 (3n-18)^{[(n-4)(47n^2-94n+15)/6]}&
   (4n-34)^{[n(n-3)/2]}\\
 (4n-30)^{[ 5(n-1)(n-2)/2]}&
  (4n-28)^{[ 5n(n-3)/2]}&
 (4n-24)^{[(19n^2-57n+20)/2]}\\
 (5n-36)^{[n-1]}&
  (5n-30)^{[5n-5]}&
   (6n-36)^{[ 1]}
\\ \hline
\end{array}$}
\caption{The eigenvalues of $A(n,6)$}\label{eig6}
\vspace{.7cm}

 \centering
{\small$\begin{array}{ll}
\hline \vspace{-.25cm} \\
(-7)^{[ n^7  - 28 n^6  + 301 n^5  - 1575 n^4  + 4179 n^3  - 5243 n^2  + 2372 n - 1]}&
(n - 19)^{[ n(n - 1) (n - 2) (n - 3) (n - 4)(n - 11)/720]}\\
(n - 18)^{[  n(n - 1) (n - 2) (n - 3)(n - 5)(n - 10)/24]}&
(n - 17)^{[ n (n - 1) (n - 2)(n - 9) (23 n^2  - 207 n + 439)/60]}\\
(n - 16)^{[  n (n - 1) (n - 8) (83 n^3  - 996 n^2  + 3691 n - 4182)/72]}&
(n - 15)^{[ n (n - 7) (11 n^4  - 154 n^3  + 739 n^2  - 1400 n + 844)/16]}\\
(n-14)^{[  (n - 1) (n - 2) (n - 6) (13 n^3  - 156 n^2  + 401 n - 10)/20]}&
(n - 13)^{[ 7n(n - 5)(n^4  - 16 n^3  + 80 n^2  - 151 n + 89)/6]}\\
(n - 12)^{[  n(n - 1) (n - 4)(13 n^3  - 208 n^2  + 1003 n - 1348)/20]}&
(n - 11)^{[ (11 n^2  - 165 n + 592) (n - 1) (n - 2) (n - 3) n/16]}\\
(n - 10)^{[n (n - 1) (n - 2)(83 n^3  - 1494 n^2  + 8593 n - 15822)/72]}&
(n - 9)^{[ n(n - 1) (n - 3) (n - 4)(23 n^2  - 299 n + 941) /60]}\\
(n - 8)^{[ n(n - 2) (n - 3) (n - 4)(n - 5)(n - 7)/24]}&
(n - 7)^{[ (n - 1) (n - 2) (n - 3) (n - 4) (n - 5) (n - 6)/720]}\\
(2n-29)^{[ n(n - 1) (n - 2) (n - 3)(n - 9)/120]}&
(2n -27)^{[ n(n - 1) (n - 2) (n - 4)(n - 8)/5]}\\
(2 n - 25)^{[ n(n - 1) (n - 7) (8 n^2  - 56 n + 89)/6]}&
(2n-23)^{[  n (n - 3)(n - 6)(17 n^2  - 102 n + 101)/8]}\\
(2 n - 22)^{[  n (n - 1) (n - 2)(11 n^2  - 132 n + 367)/10]}&
(2 n - 21)^{[ 5(n - 1) (n - 3) (n - 4)(3 n^2  - 21 n + 2)/8]}\\
(2 n - 20)^{[  5 n (n - 2) (n - 4)(n^2  - 9 n + 15)/3]}&
(2 n - 19)^{[ 7 (n - 7) (11 n^2  - 77 n + 122) n (n - 1)/20]}\\
(2 n - 18)^{[ n(n - 1) (n - 3) (n - 4)(n - 7)]}&
(2 n - 17)^{[  7 n (n - 1) (7 n^3  - 98 n^2  + 427 n - 568)/20]}\\
(2 n - 16)^{[ 5 n (n - 2) (n - 4)(n^2-9 n + 11)/3]}&
(2n -15)^{[ 7 n (n - 1) (n - 7) (3 n^2  - 21 n + 34)/8]}\\
(2 n - 14)^{[  (n - 3)(11 n^4  - 132 n^3  + 469 n^2  - 438 n + 20)/10]}&
(3n -37)^{[ n(n - 1) (n - 2) (n - 7)/24]}\\
(3 n - 34)^{[ 3n (n - 1) (n - 3) (n - 6)/4]}&
(3n-31)^{[ n(n - 5)(73 n^2  - 365 n + 382)/24]}\\
(3 n - 30)^{[ n (n - 1) (n - 2) (n - 7)/4]}&
(3 n - 29)^{[ 7 n (n - 1) (n - 3) (n - 6)/4]}\\
(3 n - 28)^{[ 5 (n - 1) (n - 2) (n - 3) (n - 4)/6]}&
(3 n - 27)^{[  5 n(n - 3)(5 n^2  - 35 n + 44)/4]}\\
(3n-26)^{[ 7n (n - 1) (n - 3) (n - 6) n/4]}&
(3 n - 25)^{[ 7 n (n - 1) (n^2  - 9 n + 19)/2]}\\
(3 n - 24)^{[ 5 n (n - 2) (n - 3)(n - 5)/2]}&
(3n-22)^{[ 7n (n - 1) (n - 2)  (n - 7)/12]}\\
(3n-23)^{[ 35 n(n - 1) (n - 3) (n - 6)/8]}&
(3 n - 21)^{[ (75n^4-750n^3 + 2233n^2  - 1958n + 120)/8]}\\
(4 n-43)^{[ n (n - 1) (n - 5)/6]}&
(4 n - 39)^{[ 2 n (n - 2) (n - 4)]}\\
(4 n - 36)^{[ n (n - 1) (n - 5)]}&
(4 n - 35)^{[ 5 (n - 1) (n - 2) (n - 3)/2]}\\
(4n-34)^{[  14 n (n - 2) (n - 4)/3]}&
(4 n - 32)^{[ 5 n (n - 2) (n - 4)]}\\
(4n-31)^{[  7 n (n - 1) (n - 5)/3]}&
(4 n - 28)^{[ (52n^3-312 n^2+425n^2 - 60)/3]}\\
(5 n - 47)^{[ n(n - 3)/2]}&
(5n-42)^{[ 3 (n - 1) (n - 2)]}\\
(5 n - 40)^{[ 3n(n - 3)]}&
(5 n - 35)^{[ (29n^2-87n+30)/2]}\\
(6n-49)^{[ n - 1]}&
(6 n - 42)^{[ 6(n -1)]}\\
(7 n - 49)^{[ 1]} \\ \hline
\end{array}$}
\caption{The eigenvalues of $A(n,7)$}\label{eig7}
\end{table}

\subsection{Connection with the Johnson graph}

We first recall the eigenvalues of the Johnson graph (see \cite[p. 179]{bh}).
\begin{lem} The eigenvalues of $J(n,k)$ are
  $$(k-i)(n-k-i)-i~~\hbox{with multiplicity}~{n\choose i}-{n\choose i-1}, ~~\hbox{for}~~i=0,\ldots,k.$$
\end{lem}

Let $\alpha_1,\ldots,\alpha_{n\choose k}$ be all the $k$-subsets of $[n]$. Let $V_i$ be the set of all permutations of $\alpha_i$.
If $|\alpha_i\cap \alpha_j|=k-1$, then each $\pi\in V_i$ is adjacent to exactly one $q\in V_j$ and if $|\alpha_i\cap \alpha_j|<k-1$, then no vertex of $V_i$ has a neighbor in $V_j$.
It follows that $(V_1,\ldots,V_n)$ is an equitable partition of $A(n,k)$ where its quotient matrix is the adjacency matrix of $J(n,k)$. So we come up with the following.

\begin{pro}\label{john} For every $i=0,\ldots,k$, $(k-i)(n-k-i)-i$ is an eigenvalue of $A(n,k)$ with multiplicity at least
${n\choose i}-{n\choose i-1}.$
\end{pro}

\subsection{The smallest eigenvalue}

From Proposition~\ref{john} it follows that $-k$ is an eigenvalue of $A(n,k)$ with multiplicity at least ${n\choose k}-{n\choose k-1}$.
In this section we establish that $-k$ is indeed the smallest eigenvalue, but with a much larger multiplicity.

\begin{thm} If $n\ge2k$, then $-k$ is the smallest eigenvalue of $A(n,k)$ with multiplicity at least $n(n-1)\cdots(n-k+2)(n-2k+1)$.
\end{thm}
\begin{proof}{
Consider the complete $k$-partite graph $K_{n,\ldots,n}$ and let $a_{i1},\ldots,a_{in}$ be the vertices of the $i$th part for $i=1,\ldots,k$.
We remove the edges $\{a_{ij}a_{rj}\mid 1\le j\le n, 1\le i,r\le k\}$ and denote the resulting graph by $H_{n,k}$.
A set of vertices $\{a_{1j_1},\ldots,a_{kj_k}\}$ forms a $k$-clique in $H_{n,k}$ if and only if $(j_1,\ldots,j_k)$ is a $k$-permutation of $[n]$.
Hence there is a one-to-one correspondence between the $k$-cliques of $H_{n,k}$ and the $k$-permutations of $[n]$.
For a $k$-permutation $\pi$ we denote the corresponding $k$-clique by $C(\pi)$.
Now, two $k$-permutations $\pi_1$ and $\pi_2$ are adjacent in $A(n,k)$ if and only if $|C(\pi_1)\cap C(\pi_2)|=k-1$.
Let $M$ denote the incidence matrix of $(k-1)$-cliques versus $k$-cliques of $H_{n,k}$, that is the rows and columns of
$M$ are indexed by the $(k-1)$-cliques and the $k$-cliques of $H_{n,k}$, respectively, where $M(C,C')=1$ if $C\subset C'$ and
$M(C,C')=0$ otherwise.
It is straightforward to see that the adjacency matrix of $A(n,k)$ is equal to $M^\top M-kI$.
It follows that  $-k$ is the smallest eigenvalue of $A(n,k)$.
As $M$ has $n(n-1)\cdots(n-k+2)k$ rows and $n(n-1)\cdots(n-k+1)$ columns, it follows that the multiplicity of $-k$ as an eigenvalue of $A(n,k)$ is
 at least $n(n-1)\cdots(n-k+2)(n-2k+1)$.
}\end{proof}

We close the paper with some open problems on the eigenvalues of the arrangement graphs.

The main problem we would like to put forward is the following:

\noindent{\bf Problem 1.} What are the eigenvalues of the arrangement graph $A(n,k)$?

Our observations allow us to narrow down Problem 1 to the following more specific conjectures:

\noindent{\bf Conjecture 2.} The eigenvalues of the arrangement graphs $A(n,k)$ consist entirely of integers.\footnote{Recently, Conjecture~2 was settled affirmatively in \cite{cgw}.}

\noindent{\bf Conjecture 3.} For any integer $k$, there is an integer $n_0$ such that for all $n\ge n_0$, $-k$ is the only negative eigenvalue of $A(n,k)$.

\section*{Acknowledgments}
 This project is supported by University of Malaya (OCAR Canseleri B21610). This work was completed while the second author was visiting Institute of Mathematical Sciences, University of Malaya.
He thanks the institute for their hospitality and support. He also would like to thank Ali Mohammadian for introducing to him the arrangement graphs and pointing out the problem of investigating their eigenvalues. The research of the second author was partially supported by a grant from IPM (No. 92050114).
The authors are grateful to an anonymous referee for several helpful comments and suggestions.

\end{document}